\documentclass{amsart}
\usepackage{latexsym, amssymb, amsmath, amsthm}
\usepackage[T1]{fontenc}
\usepackage{lmodern}
\usepackage{enumerate}
\usepackage{mathabx}
\usepackage{csquotes}
\usepackage[mathscr]{euscript}
\usepackage{parskip}
\usepackage{thmtools, thm-restate}
\usepackage{bussproofs}
\usepackage{xcolor}
\usepackage{todonotes}
\usepackage{array}

\makeatletter
\newlength\knuthian@fdfive
\def\mathpal@save#1{\let\was@math@style=#1\relax}
\def\utilde#1{\mathpalette\mathpal@save
              {\setbox124=\hbox{$\was@math@style#1$}%
\setbox125=\hbox{$\fam=3\global\knuthian@fdfive=\fontdimen5\font$}
\setbox125=\hbox{$\widetilde{\vrule height 0pt depth 0pt width \wd124}$}%
               \baselineskip=1pt\relax
               \lineskiplimit=\z@\relax
               \lineskip=1pt\relax
               \vtop{\copy124\copy125\vskip -\knuthian@fdfive}}}
\makeatother

\DeclareMathOperator{\base}{\mathsf{B\Sigma_1}}

\usepackage{pifont}

\declaretheorem[numberwithin=section]{theorem}
\newtheorem{lemma}[theorem]{Lemma}

\newtheorem{proposition}[theorem]{Proposition}
\newtheorem{question}[theorem]{Question}
\newtheorem*{claim}{Claim}

\theoremstyle{definition}
\newtheorem{definition}[theorem]{Definition}
\newtheorem*{definition*}{Definition}
\theoremstyle{remark}
\newtheorem{remark}[theorem]{Remark}

\newcommand{\form}{\mathsf{Form}}
\newcommand{\num}[1]{\underline{#1}}

\title{Descending sequences in reflection hierarchies}
\author{Mateusz \L{}e\l{}yk}
\address{Faculty of Philosophy, University of Warsaw}
\email{mlelyk@uw.edu.pl}
\author{James Walsh}
\address{Department of Philosophy, New York University}
\email{jmw534@nyu.edu}
\thanks{Thanks to Patrick Lutz, Leszek Ko\l{}odziejczyk, Antonio Montalb\'an, and Albert Visser.}

\begin{document}

\maketitle

\begin{abstract}
    There is no recursively enumerable sequence of sufficiently strong 2-consistent r.e.\ theories such that each proves the $2$-consistency of the next. Montalb\'an and Shavrukov independently asked whether this result generalizes to $0'$-recursive sequences. We consider a general version of this problem: For arbitrary $n$, for which complexity classes $\Gamma$ are there $\Gamma$-definable sequences of $n$-consistent r.e.\ theories each of which proves the $n$-consistency of the next? The answer to this question depends not only on $n$ and $\Gamma$ but also on the manner in which sequences are encoded in arithmetic. We provide positive answers for certain encodings and negative answers for others.
\end{abstract}

\section{Introduction}

G\"{o}del's second incompleteness theorem states that no reasonable consistent theory proves its own consistency.\footnote{By a reasonable theory we mean an effectively axiomatized theory that interprets a modicum of arithmetic.}  This also prohibits a certain kind of circularity with respect to provable consistency. That is, there are no consistent reasonable theories $T_0,T_1,\dots,T_n$ such that $T_0$ proves the consistency of $T_1$, $T_1$ proves the consistency of $T_2$,\dots, and $T_n$ proves the consistency of $T_0$. Otherwise, $T_0$ could chain each of these consistency proofs together to prove its own consistency, contradicting G\"{o}del's theorem.

Nevertheless, G\"{o}del's theorem does not prohibit all \emph{ill-foundedness} with respect to provable consistency. There are infinite sequences $(T_n)_{n\in\mathbb{N}}$ of theories, each of which is reasonable and consistent, such that for each $n\in\mathbb{N}$, $T_n$ proves the consistency of $T_{n+1}$. There are many means of producing such descending sequences, but they generally involve self-reference or other \emph{ad hoc} tricks. It is not easy---if it is even possible---to define descending sequences in consistency strength consisting only of natural theories.

Why is it so difficult to find descending sequences of reasonable consistent theories? One potential explanation is that any such sequence must be recursion-theoretically complicated. Gaifman once asked whether any such sequence could be recursive. The following answer is recorded in \cite{lindstrom1997aspects, smorynski1985self}. See also \cite[\textsection 6]{montalban2019inevitability}.

\begin{theorem}[H.\ Friedman, Smory\'nski, Solovay]\label{fss-theorem}
    There is a recursive sequence $(T_n)_{n\in\mathbb{N}}$ of sound r.e.\ extensions of $\mathsf{PA}$ such that for each $n\in\mathbb{N}$, $T_n \vdash \mathsf{Con}(T_{n+1}).$
    
    However, there is no recursively enumerable sequence $(T_n)_{n\in\mathbb{N}}$ of consistent r.e.\ extensions of $\mathsf{PA}$ such that $\mathsf{PA}\vdash \forall x \; \mathsf{Pr}_{T_x} \big( \mathsf{Con}(T_{x+1}) \big).$
\end{theorem}
That is, though there are recursive descending sequences in the consistency strength hierarchy, all such sequences enjoy some high degree of non-uniformity. The interplay between uniformity considerations and the existence of descending sequences in reflection hierarchies is a recurring theme in this paper.

Surprisingly, the situation changes when one considers stronger reflection principles. Recall that a theory $T$ is $n$-consistent if $T$ is consistent with the true $\Pi_n$ theory of arithmetic. We can formalize the $n$-consistency of an r.e.\ theory $T$ using a single $\Pi_{n+1}$ sentence in the language of first-order arithmetic. 
$$n\mathsf{Con}(T) := \forall \varphi \in \Pi_n \big( \mathsf{True}_{\Pi_n}(\varphi) \to \mathsf{Con}(T+\varphi)\big)$$
Note that $n$-consistency is $\mathsf{EA}$-provably equivalent to uniform $\Pi_{n+1}$-reflection and also to uniform $\Sigma_{n}$-reflection (see \textsection \ref{preliminaries}).

The following result comes from previous work of Pakhomov and the second-named author \cite[Theorem 3.4]{pakhomov2021reflection}.

\begin{theorem}[Pakhomov--Walsh]\label{original-descending}
    There is no r.e.\ sequence $(T_n)_{n\in\mathbb{N}}$ of 2-consistent r.e.\ extensions of $\mathsf{B\Sigma_1}$ such that for each $n\in\mathbb{N}$, $T_n\vdash  2\mathsf{Con}(T_{n+1}).$\footnote{\cite[Theorem 3.4]{pakhomov2021reflection} is stated for extensions of $\mathsf{EA}$. However, this is because that paper uses a slightly non-standard formalization of provability, sometimes called \emph{smooth provability}. For the proof to work with the standard provability predicate requires $\mathsf{B\Sigma_1}$.}
\end{theorem}

Theorem \ref{original-descending} does not generalize to 1-consistency; there exists a recursively enumerable sequence of 1-consistent r.e.\ extensions of $\mathsf{B\Sigma_1}$ each of which proves the 1-consistency of the next \cite[Theorem 3.8]{pakhomov2021reflection}.

Upon seeing Theorem \ref{original-descending}, Montalb\'{a}n and Shavrukov independently asked (in private communication) whether there exists a $0'$-recursive or $0''$-recursive sequence $(T_n)_{n\in\mathbb{N}}$ of 2-consistent r.e.\ theories each of which proves the 2-consistency of the next; see \cite[Question 3.10]{pakhomov2021reflection}. These questions introduce subtleties about what it means for each theory to prove the 2-consistency ``of the next.'' The theories under consideration are all $\Sigma_1$-complete, so they have a firm grasp on recursively enumerable sequences. More precisely, given a $\Sigma_1$ presentation of a sequence and given any number $n$, a $\Sigma_1$-complete theory can correctly identify the $n$th element of the sequence. There are multiple ways of encoding $\Sigma_1$-definable sequences of theories in arithmetic, but choosing between them is often not important for this reason. The same does not hold for sequences that are not $\Sigma_1$-definable. So when we turn our attention to $0'$-recursive or $0''$-recursive sequences, choosing between the multiple ways of encoding sequences \emph{does} matter.

\subsection{Encoding Sequences}

Before summarizing the main results of the paper, it is worth saying a word about the different approaches we consider to encoding sequences. In this introduction we provide only rough intuitive glosses; for precise definitions see \textsection \ref{preliminaries}. Here are two ways to understand the claim that a binary formula $\tau$ defines a sequence $(T_n)_{n\in\mathbb{N}}$ of theories.
\begin{enumerate}
\item   Slice encoding: For each $n$, $T_n=\{k : \mathbb{N}\vDash \tau(n,k) \}$.
\item   Index encoding: For each $n$, $k$ is the index of the theory $T_n$, where $k$ is the unique number $m$ such that $\mathbb{N}\vDash \tau(n,m)$.
\end{enumerate}
If one adopts the index encoding, one must then say what it is for each theory to prove the $m$-consistency of the next. Here are two ways to understand this latter claim:
\begin{enumerate}
\item   Uniform: For each $n$, $T_n$ proves the formula $\forall x \big(\tau(n+1,x) \to m\mathsf{Con}(x)\big)$.
\item   Non-uniform: For each $n$ and $k$, if $\mathbb{N}\vDash \tau(n+1,k)$, then $T_n$ proves the formula $m\mathsf{Con}(k)$.
\end{enumerate}

\subsection{A Summary of the Main Results}

The main question of this paper---informally stated---is: 
\begin{question}\label{main-question}
    Under what conditions are there sequences of $m$-consistent theories each of which proves the $m$-consistency of the next?
\end{question} 
We call such a sequence---that is, a sequence of $m$-consistent theories each of which proves the $m$-consistency of the next---\emph{an $m$-sequence}. What exactly an $m$-sequence is depends on whether one encodes sequences via slices or indices and whether one interprets the question uniformly or non-uniformly. This leads to various refinements of Question \ref{main-question}.

Table 1 summarizes the main results of the paper. Some results mention the PPI condition; PPI stands for \emph{pointwise provably inhabited} (PPI). A formula $\tau$ encoding a sequence $(T_n)_{n\in\mathbb{N}}$ is PPI if for each $n$, $T_n\vdash \exists x \tau(n+1,x)$. Note that by a ``positive result'' we mean one establishing the existence of a descending sequence and by a ``negative result'' we mean one barring the existence of descending sequences.

\begin{table}[htbp]
\centering
\setlength{\arrayrulewidth}{0.5mm}
\begin{tabular}{|>{\centering\arraybackslash}m{2.5cm}|m{5cm}|m{5cm}|}
\hline
\textbf{Approach} & \textbf{Positive Results} & \textbf{Negative Results} \\
\hline
\textbf{Slices} &
For all $m\geq 0$, there is a $\Sigma_m$-definable $m$-sequence over $\base$. (Theorem \ref{thm::SigmaDefUniSeq}).
For all $m\geq 1$, there exists a $\Pi_{m-1}$-definable $m$-sequence over $\mathsf{B\Sigma_m}$. (Theorem \ref{craig-trick-theorem})&
For $m\geq 2$, there is no $\Sigma_{m-1}$-definable $m$-sequence over $\mathsf{B\Sigma_1}$. (Theorem \ref{generalization})\\
\hline
\textbf{Indices}

(Uniform) &
For $m\geq 0$, there is a $\Pi_m$-definable uniform $m$-sequence over $\mathsf{B\Sigma_m}$.
 (Theorem \ref{uni-index-pos}) &
    For $m\geq 2$, there is no $\Sigma_m$-definable uniform $m$-sequence over $\mathsf{B\Sigma_1}$.
 (Theorem \ref{new-uni-negative}) \\

\hline
\textbf{Indices}

(Non-uniform) &
For $m\leq 1$, there is a $\Sigma_0$-definable PPI $m$-sequence over $\mathsf{EA}$.  (Theorem \ref{fss-theorem}, \cite[Theorem 3.8]{pakhomov2021reflection}) &
For $m\geq 2$, there is no $\Sigma_m$-definable PPI $m$-sequence over $\mathsf{B\Sigma_1}$. (Theorem \ref{main-negative}) \\
\hline
\end{tabular}
\caption{Summary of Results}
\label{tab:reflection-hierarchies}
\end{table}

In fact, some of the results are established in a more general form. For instance, we prove them not only relative to a fixed base theory (e.g., $\base$) but to all sufficiently sound extensions of that base theory.

Let's make two observations about this table. First, the positive result for slices cannot be improved by showing that there exists a $\Sigma_{m-1}$-definable $m$-sequence. Indeed, in the case $m=2$, the existence of such a sequence would contradict the negative result. So the positive result for slices is optimal. Second, note that the positive result for indices (non-uniform) precludes us from strengthening the negative result to all $m$.

\subsection{Outline of the Paper}

Here is our plan for the rest of the paper. In \textsection \ref{preliminaries} we cover some preliminaries about encodings and reflection principles. In \textsection \ref{uniform-section} we work with the slice encoding, culminating with a proof of Theorem \ref{craig-trick-theorem}. In \textsection \ref{non-uniform-section} we work with the index encoding. We prove Theorem \ref{uni-index-pos}, Theorem \ref{new-uni-negative}, and Theorem \ref{main-negative}.

\section{Preliminaries}

\subsection{Reflection Principles}\label{preliminaries}

Our primary objects of study are reflection principles, which are generalizations of consistency statements. The reflection principles we consider in this paper have various formulations.

\begin{definition}
    A theory $T$ is k-consistent if $T$ is consistent with each $\Pi_k$ truth.
\end{definition}

\begin{definition}
    A theory $T$ is $\Gamma$-sound if every $\Gamma$ theorem of $T$ is true.
\end{definition}

There are many ways of expressing the $\Gamma$-soundness of one theory within another. However, let us make an important note about this. For a theory $T$, any \emph{formal} statement to the effect that $T$ is $\Gamma$-sound will rely on some intensional presentation of the theory $T$. We adopt the policy in this paper that theories \emph{are} elementary formulas defining sets of axioms. That is, theories are individuated in a very fine-grained, syntactic manner. Note that Craig's trick demonstrates that every r.e. set of formulas has an elementary presentation up to logical equivalence.

The complexity classes we are interested in ($\Pi_n$, $\Sigma_n$) have definable truth-predicates within the language of arithmetic. For these complexity classes we may axiomatize the $\Gamma$-soundness of $T$ with a single formula:
$$\mathsf{RFN}_\Gamma(T) := \forall \varphi \in \Gamma\big( \mathsf{Pr}_T(\varphi) \to \mathsf{True}_\Gamma(\varphi)\big).$$
$\mathsf{EA}$ proves that  the single formula $\mathsf{RFN}_\Gamma(T)$ has the same theorems as the schema consisting of all formulas 
$$\forall \vec{x} \big(\mathsf{Pr}_T(\varphi(\vec{x})) \to \varphi(\vec{x})\big)$$
for $\varphi\in \Gamma$. The connection between $k$-consistency and $\Gamma$-soundness is codified in the following proposition:
\begin{proposition}
    Provably in $\mathsf{EA}$, for all $T$ extending $\mathsf{EA}$, the following are equivalent:
    \begin{enumerate}
        \item $T$ is n-consistent.
        \item $T$ is $\Sigma_n$-sound.
        \item $T$ is $\Pi_{n+1}$-sound.
    \end{enumerate}
\end{proposition}

Of course, it is impossible to express the \emph{global} reflection principle of a theory $T$---according to which every theorem of $T$ is true---without an extra truth-predicate. In this paper we will at some points need to discuss the following schematic approximation of the global reflection principle:
\begin{definition}
The \emph{uniform reflection schema} $\mathsf{RFN}(T)$ for $T$ consists of all formulas
$$\forall \vec{x} \big(\mathsf{Pr}_T(\varphi(\vec{x})) \to \varphi(\vec{x})\big)$$
where $\varphi$ is a formula in the language of $\mathsf{EA}$.
\end{definition}

Before continuing, we record a trivial lemma.

\begin{lemma}\label{trivial-lemma}
Suppose that $T$ is $m$-consistent. Then $T$ is consistent with any true $\Sigma_{m+1}$ sentence.
\end{lemma}

\begin{proof}
    Suppose that $T$ is $m$-consistent. Now let $\varphi \in \Sigma_{m+1}$ be true. Then $\varphi$ has the form $\exists x \pi(x)$ where $\pi(x)$ is a $\Pi_m$ formula. Since $\varphi$ is true, then $\pi(n)$ is true for some $n$. Since $T$ is $m$-consistent, $T+\pi(n)$ is consistent. But $\pi(n)$ implies $\varphi$, so $T+\varphi$ is consistent. 
\end{proof}

We will also need to deal with theories that are defined by \emph{iterating} reflection principles.
\begin{definition}
    \label{iterated-reflection}
The iterated reflection principle $m\mathsf{Con}^{k+1}(T)$ is the statement $m\mathsf{Con}(T+m\mathsf{Con}^k(T))$.
\end{definition}

\subsection{The Index Encoding}

First, we consider versions of the problem wherein a sequence of r.e.\ theories is encoded as a sequence of indices of Turing machines numerating those theories. To state this version, we introduce some terminology. First, we must say it is for a formula to define a sequence.
\begin{definition}
    A binary formula $\tau$ \emph{defines an index sequence} if there is a sequence $(T_n)_{n\in\mathbb{N}}$ of r.e.\ theories such that, for each $n$, $k$ is the index of the theory $T_n$, where $k$ is the unique number $m$ such that $\mathbb{N}\vDash \tau(n,m)$.\footnote{There are various ways of understanding what an index of a recursively axiomatized theory is, but they will not matter for us. An index could be an index of a Turing machine that numerates the axioms of the theory. It could be a $\Delta_0$ formula that bi-numerates the axioms of the theory.}
\end{definition}
That is, a formula defines an index sequence just in case it defines a sequence of indices of theories.

\begin{remark}
    Note that $T_n$ may be recursively axiomatized even when the formula $\tau(n,x)$ has high quantifier complexity. For instance, $\tau$ may define a sequence of r.e.\ theories even if each formula $\tau(n,x)$ is not $\Sigma_1$.
\end{remark}

We are particularly interested in sequences of theories meeting a certain constraint, namely:
\begin{center}
    $(\star)$\quad Each entry proves a reflection principle for the next theory.
\end{center}
Given the index encoding of theories, there are multiple ways of interpreting the constraint $(\star)$. 

\subsubsection{A Uniform Interpretation}

We first provide a \emph{uniform} interpretation of constraint $(\star)$. According to the uniform interpretation, the $n$th theory in the sequence defined by $\tau$ must prove the claim:
$$\forall x \big( \tau(n+1,x) \to m\mathsf{Con}(x)\big)$$
where $m\mathsf{Con}$ is the relevant reflection principle. That is, each theory proves reflection for the next by recourse to the formula $\tau$ defining the sequence. 

\begin{definition}
    Let $(T_n)_{n\in\mathbb{N}}$ witness that $\tau$ defines an index sequence. We say that $\tau$ is a \emph{uniform $m$-sequence over $T$} if each of the following holds:
    \begin{enumerate}
        \item Each $T_n$ is an $m$-consistent r.e. extension of $T$.
        \item For each $n$, $T_n\vdash \exists x \; \tau(n+1,x)$.
        \item For each $n$, $T_n\vdash \forall x \big( \tau(n+1,x) \to m\mathsf{Con}(x)\big)$.
    \end{enumerate}
\end{definition}

Our first refinement of Question \ref{main-question} concerns the existence of uniform $m$-sequences, where sequences are encoded as index sequences.

\begin{question}\label{universal-question}
    Adopt the index approach to encoding sequences. For which choices of $\langle T,\Gamma,m\rangle$ is there a $\Gamma$-definable uniform $m$-sequence over $T$?
\end{question}

Theorem \ref{original-descending} yields a negative answer for the choice $\langle \mathsf{B\Sigma_1}, \Sigma_1, 2\rangle$. A generalization of this result shows that, given the index approach to encoding sequences, for each $m\geq 2$, there is no $\Sigma_m$-definable uniform $m$-sequence over $\mathsf{B\Sigma_1}$ (Theorem \ref{new-uni-negative}).

\subsubsection{A Non-uniform Interpretation} 

We now describe a \emph{non-uniform} interpretation of constraint $(\star)$. According to the non-uniform interpretation,the $n$th theory in the sequence defined by $\tau$ must prove the claim:
$$ m\mathsf{Con}(k)$$
where $\mathbb{N}\vDash \tau(n+1,k)$. That is, each theory proves reflection for the next index \emph{directly}, without recourse to the formula $\tau$ defining the sequence.

\begin{definition}
         Let $(T_n)_{n\in\mathbb{N}}$ witness that $\tau$ defines an index sequence. We say that $\tau$ is a \emph{non-uniform $m$-sequence over $T$} if each of the following holds:
        \begin{enumerate}
            \item Each $T_n$ is an $m$-consistent r.e.\ extension of $T$.
            \item For each $n$, for the unique $k$ such that $\mathbb{N}\vDash \tau(n+1,k)$, $T_n\vdash m\mathsf{Con}(k).$
        \end{enumerate}
\end{definition}

Using this terminology, we state the non-uniform version of the question as follows:
\begin{question}\label{non-uniform-question}
    Adopt the index approach to encoding sequences. For which choices of $\langle T,\Gamma,m\rangle$ is there a $\Gamma$-definable non-uniform $m$-sequence over $T$?
\end{question}

\begin{remark}
    Note the distinction between this question and the previous one. Suppose that there is a $\Gamma$-definable non-uniform 2-sequence. This means that some $\Gamma$ formula $\tau$ defines a sequence $(\tau_n)_{n\in\mathbb{N}}$ of indices of 2-consistent extensions $(T_n)_{n\in\mathbb{N}}$ of $\mathsf{B\Sigma_2}$ such that for each $n\in\mathbb{N}$, $T_n\vdash  2\mathsf{Con}(k)$, where $k$ is the index of $T_{n+1}$. Note that the statement $2\mathsf{Con}(k)$ that $T_n$ proves might not make any use of the formula $\tau$ that defines the sequence. This is why we call this version of the problem \emph{non-uniform}.
\end{remark}

The distinction between the uniform and non-uniform versions of the problems matters only when we consider non-r.e.\ sequences. Hence, Theorem \ref{original-descending} again yields a negative answer to Question \ref{non-uniform-question} for the choice $\langle \mathsf{B\Sigma_1}, \Sigma_1, 2\rangle$.

For the non-uniform version of the question, we prove a strong analogue of the H.\ Friedman, Smory\'nski, Solovay theorem. In particular, we answer this question in the special case of sequences that are {pointwise provably inhabited}.

\begin{definition}
     We say that a sequence $\sigma$ enoding $(T_n)_{n\in\omega}$ is {pointwise provably inhabited} if for each $n\in\mathbb{N}$, $T_n\vdash \exists x \sigma(n+1,x)$.
\end{definition}

One of our main theorems (Theorem \ref{main-negative}) shows that, given the index approach to encoding sequences, for each $m$, there is no $\Sigma_m$-definable non-uniform $m$-sequence that is pointwise provably inhabited. Note that this result does not \emph{merely} banish provably-\emph{descending} non-uniform 2-sequences but banishes all pointwise provably \emph{inhabited} non-uniform 2-sequences.

\subsection{The Slice Encoding}

In the previous subsection we encoded a sequence of theories as a sequence of indices numerating those theories. There is another common approach to encoding sequences, namely, to encode each theory in the slices of a binary relation.
\begin{definition}
    A binary formula $\tau$ \emph{defines a slice sequence} if there is a sequence $(T_n)_{n\in\mathbb{N}}$ such that for every $n\in\mathbb{N}$, $T_n=\{k : \mathbb{N}\vDash \tau(n,k) \}$.
\end{definition}

If we adopt the slice encoding, then there is a natural uniform version of our problem, but there is no clear natural non-uniform version. Let's introduce some terminology needed to state the uniform problem.

\begin{definition}
         Let $(T_n)_{n\in\mathbb{N}}$ witness that $\tau$ defines a slice sequence. We say that $\tau$ is a \emph{uniform $m$-sequence over $T$} if each of the following holds:
        \begin{enumerate}
            \item  Each $T_n$ is a $m$-consistent r.e.\ extension of $T$.
            \item For each $n$, $T_n\vdash \exists x \tau(n+1,x)$
            \item For each $n$: 
    $$T_n\vdash \forall x \forall y \Big( x=\ulcorner \bigwedge \{ \varphi_ z \mid z\leq y \wedge  \tau(n+1,z) \}\urcorner \to m\mathsf{Con}(x)   \Big).$$
        \end{enumerate}
\end{definition}

\begin{remark}
    One might wonder why we demand that $T_n$ proves $m\mathsf{Con}$ of all conjunctions of this form. The issue is that the $\tau$ might numerate formulas in such a way that each $\tau(n+1,k)$ is $m$-consistent yet their union is jointly $m$-inconsistent. Of course, any failure of $m$-consistency is witnessed by some finite conjunction of formulas. Hence, to prove that the theory numerated by $\tau(n+1,\cdot)$ is $m$-consistent, it suffices to show that for every $k$, the conjunction:
$$\bigwedge \{ \varphi_ z \mid z\leq k \wedge  \tau(n+1,z) \}$$
is $m$-consistent.
\end{remark}

This leads to one more version of our motivating question.

\begin{question}\label{slice-question}
    Adopt the slice approach to encoding sequences. For which choices of $\langle T,\Gamma,m\rangle$ is there a $\Gamma$-definable uniform $m$-sequence over $T$?
\end{question}

Once again, Theorem \ref{original-descending} shows that there is no $\Sigma_1$-definable uniform $2$-sequence over $\mathsf{B\Sigma_1}$. One of our main results (Theorem \ref{craig-trick-theorem}) provides a counter-point to this by producing descending sequences that are not much more complicated. In particular, if we adopt the slice approach to encoding sequences, then there is a $\Pi_{m}$-definable uniform $m+1$-sequence over $\mathsf{B\Sigma_{m+1}}$.


\section{The Slice Encoding}\label{uniform-section}

In this section we adopt the slice encoding. The ultimate goal of this section is to show that there is a $\Pi_{m}$-definable uniform $m+1$-sequence over $\mathsf{B\Sigma_{m+1}}$. The proof of this result requires a few tricks, so we build up to it in stages. We will prove a few weak versions of this theorem first since their proofs are easier to understand. Afterward we will use some tricks to modify these into a proof of the desired result.

Note that all theorems stated in this section are stated relative to the slice encoding of sequences.

\subsection{Visser's Technique}

The proof of this section's main theorem deploys a technique due to Albert Visser \cite{visser1988descending}. Visser defines a recursive sequence $(T_n)_{n\in\mathbb{N}}$ of consistent extensions of $\mathsf{PA}$ such that for each $n\in\mathbb{N}$, $T_n$ proves every instance of uniform reflection for $T_{n+1}$. The theories that Visser produces are not $1$-consistent, so this result does not contradict Theorem \ref{original-descending}.

We reproduce Visser's theorem here with a slightly different proof (in particular, we use self-reference instead of the recursion-theoretic fixed point theorem). Whereas the original proof uses the recursion theorem, we use self-reference in arithmetic. Moreover, instead of $\mathsf{PA}$ we change the base system to $\mathsf{B\Sigma_1}$.

 \begin{theorem}
     There is a recursive sequence $(T_n)_{n\in\mathbb{N}}$ of consistent extensions of $\mathsf{B\Sigma_1}$ such that for each $n\in\mathbb{N}$, $T_n$ proves every instance of the uniform reflection schema for $T_{n+1}$.
 \end{theorem}

\begin{proof}
Using the G\"odel--Carnap self-reference lemma, we get a sequence $(T_n)_{n<\omega}$ of theories such that 
{\small \begin{equation*}
T_n = \left\{
        \begin{array}{ll}
            \mathsf{RFN}(\mathsf{B\Sigma_1}+T_{n+1}) & \quad \forall k \leq n \; \neg \mathsf{Prf}_\mathsf{ZF}(k,\bot) \\
            \mathsf{B\Sigma_1} & \quad \text{otherwise}.
        \end{array}
    \right.
\end{equation*}}

More formally, we get a binary formula $T$ such that $\mathsf{B\Sigma_1}$ proves $T(n,x)$ iff $$x = \ulcorner\mathsf{B\Sigma_1}\urcorner \vee (\forall k \leq n \; \neg \mathsf{Prf}_\mathsf{ZF}(k,\bot) \wedge x \in \mathsf{RFN}(\{y:T(n-1,y)\})).$$
Informally speaking, we think of $T(n,x)$ as saying ``$x$ is an axiom of theory $T_n$.''

\emph{Reason in $\mathsf{ZF}$:} Suppose $\neg \mathsf{Con}(\mathsf{ZF})$. Let $n$ be the least proof of $\bot$ in $\mathsf{ZF}$. Then $T_{n+1}=\mathsf{B\Sigma_1}$. So $T_0=\mathsf{RFN}^n(\mathsf{B\Sigma_1})$. The latter theory is consistent. So $\mathsf{Con}(T_0).$

This argument shows that $\mathsf{ZF}$ proves $\neg \mathsf{Con}(\mathsf{ZF})$ implies $\mathsf{Con}(T_0)$. So $\mathsf{ZF}$ proves $\neg \mathsf{Con}(T_0)$ implies $\mathsf{Con}(\mathsf{ZF})$. So $\mathsf{ZF}$ does not prove $\neg \mathsf{Con}(T_0)$ by G\"odel's second incompleteness theorem. But $\neg \mathsf{Con}(T_0)$ is $\Sigma_1$ and $\mathsf{ZF}$ proves all $\Sigma_1$ truths. So we can infer that $\mathsf{Con}(T_0)$ holds. An easy induction then shows that each $T_n$ is consistent. $T_0$ is consistent, which takes care of the base case. Assume $T_n$ is consistent. Since $T_n$ proves each instance of uniform reflection for $T_{n+1}$ it proves, in particular, $\mathsf{Pr}_{T_{n+1}}(\bot)\to \bot$, i.e., $\mathsf{Con}(T_{n+1})$. But every $\Pi_1$ consequence of a consistent theory is true, so $T_{n+1}$ is consistent.
\end{proof}

\subsection{The General Case}

Visser's proof  does not yield a 2-sequence, or even a 1-sequence, since the theories Visser defines are not 1-consistent. In this subsection we modify Visser's proof to produce a $\Sigma_m$-definable $m$-sequence over $\base$. Lowering the complexity of the sequence to $\Pi_{m-1}$ requires some additional tricks, so we defer it to the next subsection.

Note that in the statement of the theorem---as well as in the proof---we refer to iterated reflection principles. For the definition of iterated reflection see Definition \ref{iterated-reflection}.

\begin{theorem}\label{thm::SigmaDefUniSeq}
      Let $m\geq 0$. Let $T$ be an r.e.\ extension of $\base$ such that $\mathsf{I\Sigma}_m + \forall x \, m\mathsf{Con}^{x}(T)$ is $m$-consistent. Then there is a $\Sigma_m$-definable uniform $m$-sequence over $T$.
\end{theorem}

\begin{remark}
    The statement of Theorem \ref{thm::SigmaDefUniSeq} might appear somewhat technical. Let us emphasize that every sound theory $T$ satisfies the condition that $\mathsf{I\Sigma}_m + \forall x \, m\mathsf{Con}^{x}(T)$ is $m$-consistent. We state this theorem in this way so that we precisely calibrate the assumptions required for the proof to work. If we stated the result for all sound theories, the reader might wonder what role soundness plays in the proof. In fact, only this condition---that $\mathsf{I\Sigma}_m + \forall x \, m\mathsf{Con}^{x}(T)$ is $m$-consistent---is deployed in the proof.
\end{remark}

\begin{proof}
Note that the case $m=0$ is already covered by \ref{fss-theorem}.

Let $\mathsf{m}^{\omega}T$ abbreviate $\mathsf{I\Sigma}_m + \forall x \,m\mathsf{Con}^x(T)$. We want a sequence $(\tau_n)_{n<\omega}$ of theories such that:
{\footnotesize \begin{equation*}
\tau_n = \left\{
        \begin{array}{ll}
            T+m\mathsf{Con}(\base+\tau_{n+1}) & \quad \forall k\leq n \; \forall A\leq n \Big( A\in\Sigma_m \wedge \mathsf{Prf}_{\mathsf{m}^{\omega}T}(k,A) \to  \mathsf{True}_{\Sigma_{m}}(A)\Big) \\
            T & \quad \text{otherwise}
        \end{array}
    \right.
\end{equation*}}

Again, we use the self-reference lemma. There is a binary formula $\tau$ such that: $\base$ proves that $\tau(n,\varphi)$ is equivalent to the disjunction of the following:
\begin{enumerate}
\item $\varphi\in T$
\item both of the following hold:
\begin{enumerate}
    \item $\exists s\forall k\leq n \; \forall A\leq n \Big( A\in\Pi_{m-1} \wedge \mathsf{Prf}_{\mathsf{m}^{\omega}T}(k,\exists x A) \to  \mathsf{True}_{\Pi_{m-1}}(A(s_k))\Big)$
    \item $\varphi = \ulcorner m\mathsf{Con}(\tau_{n+1})\urcorner$
    \end{enumerate}
\end{enumerate}

Note that (2a) is slightly different from the formula that appeared in the initial piecewise definition of $T_n$. The reason for the change is this: Since the existential quantifier in the initial piecewise definition occurs in the scope of bounded universal quantifiers, $\mathsf{B\Sigma_m}$ is required to transform that formula into a $\Sigma_m$ condition. By contrast (2a) places an existential quantifier in front of the bounded universal quantifiers, so strong bounding axioms are no longer required to transform this formula into a $\Sigma_m$ formula. This existential quantifier pronounces the existence of a sequence of potential witnesses to the truth of a $\mathsf{I\Sigma}_m + \forall x \,m\mathsf{Con}^x(T)$-provable $\Pi_{m-1}$ formula. Note that this is analogous to the quantifier transformations engendered by the Axiom of Choice in set theory and second-order arithmetic.

Let's explicate (2b) in more formal terms.
$m\mathsf{Con}(\tau_{n+1})$ is the formula:
$$ \forall \psi\in \Pi_m\big( \mathsf{True}_{\Pi_m}(\psi) \to \mathsf{Con}(\tau_{n+1}+\psi)\big). $$
Where $\mathsf{Con}(\tau_{n+1}+\psi)$ is an abbreviation for:
    $$\forall x \forall y \Big( x=\ulcorner \bigwedge \{ \varphi_ z \wedge \psi \mid z\leq y \wedge  \tau(n+1,z) \}\urcorner \to \mathsf{Con}(x)   \Big).$$

\begin{claim}
    $\tau$ is $\base$-equivalent to a $\Sigma_m$ formula.
\end{claim}

Note that $\tau$ is the result of applying the Diagonal Lemma to a formula that is $\Sigma_m$. Applications of the Diagonal Lemma do not increase the quantifier complexity of formulas.

\begin{claim}
    $\tau_0$ is m-consistent.
\end{claim}

First, we need to calculate some more quantifier complexities. We have already observed that $\tau(0,\varphi)$ is ($\base$-equivalent to) $\Sigma_m$. Hence, the following is $\Pi_m$:
$$\mathsf{Con}(\tau_0+\varphi):= \forall x \forall y \Big( x=\ulcorner \bigwedge \{ \varphi_ z \wedge \varphi \mid z\leq y \wedge  \tau(0,z) \}\urcorner \to \mathsf{Con}(x)   \Big).$$

This means that the following formula is $\Pi_{m+1}:$
$$m\mathsf{Con}(\tau_0):= \forall \varphi\in \Pi_m\big( \mathsf{True}_{\Pi_m}(\varphi) \to \mathsf{Con}(\tau_0+\varphi)\big). $$
So $\neg m\mathsf{Con}(\tau_0)$ is $\Sigma_{m+1}$.

\emph{Reason in $\mathsf{m}^{\omega}T$:} Suppose $\neg m\mathsf{Con}(\mathsf{m}^{\omega}T)$. Then for some $n$ and some $\Sigma_m$ sentence $A$, $n$ encodes a $\mathsf{m}^{\omega}T$ proof of $A$ but $A$ is not a $\Sigma_m$ truth. By $\mathsf{I\Sigma_m}$ we can assume that $n$ is the least such proof of a false $\Sigma_m$-sentence. Then, by the fixed point definition, $\forall \varphi\big(\tau(n,\varphi)\leftrightarrow \phi\in T)\big)$. Moreover, by $\Sigma_m$ collection, there is $c$ such that
\[\forall k\leq n-1 \; \forall A\leq n \Big( A\in\Pi_{m-1} \wedge \mathsf{Prf}_{\mathsf{m}^{\omega}T}(k,\exists x A) \to  \mathsf{True}_{\Pi_{m-1}}(A(c_k))\Big).\]
Call the above sentence $\theta(c)$ and assume observe that $\theta(c)$ is a true $\Pi_{m-1}$ sentence.
It follows from the definition of $\tau$ and the choice of $n$ that for each $k<n$ we have
\[\forall \varphi\big(\tau(k,\varphi)\leftrightarrow \bigl(\varphi \in T \vee \varphi = \ulcorner m\mathsf{Con}(\tau_{k+1})\urcorner\bigr)\big).\]

Using the $\base$-provable properties of the fixpoint definition we observe that
\[\base+\neg A + \theta(c)\vdash \forall \phi\bigl(\tau(n,\phi)\leftrightarrow \phi\in T\bigr).\]
Hence,
\[\base +\neg A+\theta(c) + m\mathsf{Con}(T)\vdash m\mathsf{Con}(\tau_n).\]
Assume that  for $1\leq k\leq n+1$
\[\base +\neg A+\theta(c) + m\mathsf{Con}^k(T)\vdash m\mathsf{Con}(\tau_{n-k+1}).\]
Note that the statement $m\mathsf{Con}(\tau_{n-k+1})$ is very close to the theory $\tau_{n-k}$ itself. All that is missing are the axioms of $T$. Hence, it follows that
$$T + \neg A+\theta(c)+m\mathsf{Con}^k(T)\vdash \tau_{n-k}.$$
Hence, $ \base + m\mathsf{Con}\Big(T+\neg A+\theta(c)+m\mathsf{Con}^k(T) \Big)\vdash m\mathsf{Con}(\tau_{n-k})$. 

Since $\neg A$ and $\theta(c)$ are both $\Pi_m$:
$$\base\vdash \big(\neg A\wedge\theta(c)\wedge m\mathsf{Con}^{k+1}(T) \big)\rightarrow m\mathsf{Con}(\neg A+\theta(c)+T +m\mathsf{Con}^{k}(T)).$$

Combining the previous two lines, we infer: 
\[\base +\neg A+\theta(c) + m\mathsf{Con}^{k+1}(T)\vdash m\mathsf{Con}(\tau_{n-k}).\]
By $\Sigma_1$ induction we conclude that 
\[\base +\neg A+\theta(c) + m\mathsf{Con}^{n}(T)\vdash m\mathsf{Con}(\tau_{1}).\]
So $T  +\neg A+\theta(c) + m\mathsf{Con}^{n}(T)\vdash \tau_0$ and since, in $\mathsf{m}^{\omega}T+\neg A + \theta(c)$, the theory on the left hand side is $m$-consistent, so is $\tau_0$. 
 
The previous argument shows that $\mathsf{m}^{\omega}T\vdash \neg m\mathsf{Con}(\mathsf{m}^{\omega}T)\to m\mathsf{Con}(\tau_0)$. So, contraposing, $\mathsf{m}^{\omega}T\vdash\neg m\mathsf{Con}(\tau_0) \to m\mathsf{Con}(\mathsf{m}^{\omega}T)$. 

It is well-known that no consistent extension of an r.e. theory $S$ by true $\Sigma_{m+1}$ sentences proves $m\mathsf{Con}(S)$ \cite[Proposition 2.16]{beklemishev2005reflection}. So no consistent extension of $\mathsf{m}^{\omega}T$ by true $\Sigma_{m+1}$ sentences proves $\neg m\mathsf{Con}(\tau_0)$. Since $\neg m\mathsf{Con}(\tau_0)$ is a $\Sigma_{m+1}$ sentence, this means $\mathsf{m}^{\omega}T+\neg m\mathsf{Con}(\tau_0)$ is inconsistent, whence $\mathsf{m}^{\omega}T$ proves $m\mathsf{Con}(\tau_0)$. Since we assumed that $\mathsf{m}^{\omega}T$ is $m$-consistent, it follows that $\tau_0$ is $m$-consistent.

\begin{claim}
    For each $n$, $T_n\vdash m\mathsf{Con}(\tau_{n+1})$.
\end{claim}

Since, by assumption, $\mathsf{m}^{\omega}T$ is $\Sigma_m$-sound, the second disjunct of the fixed point definition of $\tau_n$ is always realized. Which is to say that, for each $n$, the formula $m\mathsf{Con}(\tau_{n+1})$ belongs to the theory $T_n$.
\end{proof}

\subsection{Craig's Trick}

Theorem \ref{thm::SigmaDefUniSeq} guarantees the existence of $\Sigma_m$-definable $m$-sequences over certain extensions of $\base$. In fact, we can secure the existence of $\Pi_{m-1}$-definable $m$-sequences by applying a variant of Craig's trick (also known as ``padding'' by recursion theorists). However, there is a cost. In particular, we can secure the existence of such sequences only over extensions of the stronger base system $\mathsf{B\Sigma}_m$. So, strictly speaking, the following theorem is neither stronger nor weaker than Theorem \ref{thm::SigmaDefUniSeq} but appears to be incomparable with it.

\begin{theorem}\label{craig-trick-theorem}
Let $m\geq 2$. Let $T$ be an r.e.\  extension of $\mathsf{B\Sigma}_m$ such that $\mathsf{I\Sigma_m} + \forall x m\mathsf{Con}^{x}(T)$ is $m$-consistent. There is a $\Pi_{m-1}$-definable uniform $m$-sequence over $T$.
\end{theorem}

\begin{proof} 
Let $\mathsf{m}^{\omega}T$ abbreviate $\mathsf{I\Sigma}_m + \forall x \,m\mathsf{Con}^x(T)$. Ultimately, we want to obtain a $\Pi_{m-1}$ formula $\tau(n,x)$ such that $\mathsf{B\Sigma}_m$ proves $\tau(n,x)$ iff:
\begin{enumerate}
\item $x\in T$ or
\item both of the following hold for some sequence $s<x$:
\begin{enumerate}
    \item $\forall k\leq n \forall A\leq n \bigl(A\in \form_{\Pi_{m-1}} \wedge \mathsf{Prf}_{\mathsf{m}^{\omega}T}(k,\exists x A) \rightarrow \mathsf{Tr}_{\Pi_{m-1}}(A(s_k))\bigr)$.
    \item $x=\ulcorner\bigwedge_{i<s} m\mathsf{Con}(\tau(n+1,x))\urcorner$.
\end{enumerate}
\end{enumerate}

Let us now explain how we find this formula $\tau(n,x)$ in detail. We start by considering the following disjunctive formula $\psi(n,x,z)$
\begin{enumerate}
\item $x\in T$ or
\item both of the following hold for some sequence $s<x$:
\begin{enumerate}
    \item $\forall k\leq n \forall A\leq n \bigl(A\in \form_{\Pi_{m-1}} \wedge \mathsf{Prf}_{\mathsf{m}^{\omega}T}(k,\exists x A) \rightarrow \mathsf{Tr}_{\Pi_{m-1}}(A(s_k))\bigr)$.
    \item $x=\ulcorner\bigwedge_{i<s} m\mathsf{Con}(t(z,n))\urcorner$.
\end{enumerate}
\end{enumerate}

In the above $\bigwedge_{i<z} B$ denotes the conjunction of length $z$ of sentence $B$ and $t(z,n)$ is a definable function which given a code $z$ of a formula $\phi(x,y)$ (with at most two free variables) returns the code of $\phi(n+1,y)$. $\psi(n,x,z)$ is of $\Sigma_m$ complexity (note that $(2a)$ is $\Pi_{m-1}$ and it is preceded by a bounded quantifier $\exists s <x$.) By applying the collection scheme, $\psi(n,x,z)$ is canonically equivalent in $\mathsf{B\Sigma}_m$ to a $\Pi_{m-1}$ formula $\tau'(n,x,z)$. That is, ${\sf B\Sigma}_m$ proves
\[\psi(n,x,z)\equiv\tau'(n,x,z).\]
We then apply the Diagonal Lemma to $\tau'(n,x,z)$ obtaining a formula $\tau(n,x)$ such that $\mathsf{B\Sigma}_m$ proves
\[\tau(n,x)\equiv\tau'(n,x,\ulcorner\tau(n,x)\urcorner).\]
(Indeed, this biconditional is provable in much weaker systems.) It is clear from the two centered biconditionals that $\tau(n,x)$ works.

\begin{claim}
    $\tau_0$ is $m$-consistent.
\end{claim}
\emph{Reason in $\mathsf{m}^{\omega}T$:} Suppose $\neg m\mathsf{Con}(\mathsf{m}^{\omega}T)$. Then for some $n$ and some $\Sigma_m$ sentence $A$, $n$ encodes a $\mathsf{m}^{\omega}T$ proof of $A$ but $A$ is not a $\Sigma_m$ truth. By $\mathsf{I\Sigma_m}$ we can assume that $n$ is the least such proof of a false $\Sigma_m$-sentence. Then, by the fixed point definition, $\forall \varphi\big(\tau(n,\varphi)\leftrightarrow \varphi\in T)\big)$. Moreover, by $\Sigma_m$ collection, there is $c$ such that
\[\forall k\leq n-1 \; \forall A\leq n-1 \Big( A\in\Pi_{m-1} \wedge \mathsf{Prf}_{\mathsf{m}^{\omega}T}(k,\exists x A) \to  \mathsf{True}_{\Pi_{m-1}}(A(c_k))\Big).\]
Call the above sentence $\theta(c)$ and assume observe that $\theta(c)$ is a true $\Pi_{m-1}$ sentence.
It follows from the definition of $\tau$ and the choice of $n$ that for each $k<n$ we have
\[\forall \varphi\big(\tau(k,\varphi)\leftrightarrow \bigl(\varphi \in T\vee \varphi = \ulcorner m\mathsf{Con}(\tau_{k+1})\urcorner\bigr)\big).\]
Note that $\tau_{k+1}$ is understood as the theory numerated by $\tau(k+1,x)$ where $k+1$ is a fixed integer and $x$ varies.

Using the $\mathsf{B\Sigma}_m$-provable properties of the fixpoint definition we observe that
\[\mathsf{B\Sigma}_m+\neg A + \theta(c)\vdash \forall \varphi\bigl(\tau(n,\varphi)\leftrightarrow \varphi \in T\bigr).\]
Hence,
\[\mathsf{B\Sigma}_m +\neg A+\theta(c) + m\mathsf{Con}(T)\vdash m\mathsf{Con}(\tau_n).\]
Assume that  for $1\leq k\leq n+1$
\[\mathsf{B\Sigma}_m +\neg A+\theta(c) + m\mathsf{Con}^k(T)\vdash m\mathsf{Con}(\tau_{n-k+1}).\]
Just as in the proof of Theorem \ref{thm::SigmaDefUniSeq}, this entails the following claim:
\[\mathsf{B\Sigma}_m +\neg A+\theta(c) + m\mathsf{Con}^{k+1}(T)\vdash m\mathsf{Con}(\tau_{n-k}).\]
By $\Sigma_1$ induction we conclude that:
\[\mathsf{B\Sigma}_m +\neg A+\theta(c) + m\mathsf{Con}^{n}(T)\vdash m\mathsf{Con}(\tau_{1}).\]
So $T +\neg A+\theta(c) + m\mathsf{Con}^{n}(T)\vdash \tau_0$ and since, in $\mathsf{m}^{\omega}T+\neg A + \theta(c)$, the theory on the left-hand side is $m$-consistent, so is $\tau_0$. 
 
The previous argument shows that $\mathsf{m}^{\omega}T\vdash \neg m\mathsf{Con}(\mathsf{m}^{\omega}T)\to m\mathsf{Con}(\tau_0)$. So, contraposing, $\mathsf{m}^{\omega}T\vdash\neg m\mathsf{Con}(\tau_0) \to m\mathsf{Con}(\mathsf{m}^{\omega}T)$. 

It is well-known that no consistent extension of an r.e.\ theory $S$ by true $\Sigma_{m+1}$ sentences proves $m\mathsf{Con}(S)$ \cite[Proposition 2.16]{beklemishev2005reflection}. So no consistent extension of $\mathsf{m}^{\omega}T$ by true $\Sigma_{m+1}$ sentences proves $\neg m\mathsf{Con}(\tau_0)$. Since $\neg m\mathsf{Con}(\tau_0)$ is a $\Sigma_{m+1}$ sentence, this means $\mathsf{m}^{\omega}T+\neg m\mathsf{Con}(\tau_0)$ is inconsistent, whence $\mathsf{m}^{\omega}T$ proves $m\mathsf{Con}(\tau_0)$. Since we assumed that $\mathsf{m}^{\omega}T$ is $m$-consistent, it follows that $\tau_0$ is $m$-consistent.

\begin{claim}
    For each $n$, $T_{n}\vdash m\mathsf{Con}(\tau_{n+1}).$
\end{claim}
Since, by assumption, $\mathsf{m}^{\omega}T$ is $m$ consistent, it is $\Sigma_m$-sound, so whenever $\mathsf{m}^{\omega}T\vdash \exists x \phi(x)$, then there is $k\in\mathbb{N}$ such that $\mathbb{N}\models \phi(\num{k})$. For any such $k$, the sentence $\bigwedge_{i<k}m\mathsf{Con}(\tau_{n+1})$ is  an axiom of $T_{n}$.
\end{proof}

\subsection{Generalizing a Negative Result}

Note that Theorem \ref{craig-trick-theorem} is the best possible result in light of Theorem \ref{original-descending}. In particular, we cannot strengthen Theorem \ref{craig-trick-theorem} to get a $\Sigma_{m-1}$-definable $m$-sequence over $T$; this would contradict Theorem \ref{original-descending} in the case $m=2$. The cases $m>2$ will likewise be unachievable, in light of the following generalization of Theorem \ref{original-descending}.

\begin{theorem}\label{generalization}
    For each $m\geq 2$, there is no $\Sigma_{m-1}$-definable uniform $m$-sequence over $\mathsf{B\Sigma_1}$.
\end{theorem}

\begin{remark}
    Note that the following proof is a relatively straightforward adaptation of the proof of \cite[Theorem 3.4]{pakhomov2021reflection}. In \textsection \ref{non-uniform-section} we provide some adaptations of this argument that are considerably less straightforward. Hence, the following proof may serve as a warmup for these results.
\end{remark}

\begin{proof}
For $m=2$, this is already covered by Theorem \ref{original-descending}. So for the rest of the proof we deal only with the case $m\geq 3$.

Suppose that there is such a sequence. We can assume that each slice is closed under conjunctions. Then $\mathsf{DS}$ is true, where $\mathsf{DS}$ is the statement 
$$\exists \sigma \in \Sigma_m \big( \theta_1(\sigma) \wedge \theta_2(\sigma) \wedge \theta_3(\sigma)\wedge \theta_4(\sigma)\big)$$ where we have:
\begin{align*}
    \theta_1(\sigma) &:= \forall x\big(\mathsf{True}_{\Sigma_{m-1}}(\sigma(0,x))\rightarrow m\mathsf{Con}(x)\big)\\
    \theta_2(\sigma) &:=  \forall x\big(\mathsf{True}_{\Sigma_{m-1}}\sigma(x,\mathsf{B\Sigma_1})\big)\\
    \theta_3(\sigma) &:=  \forall x \exists y \big(\mathsf{True}_{\Sigma_{m-1}}(\sigma(x,y)) \wedge \mathsf{Pr}_y(  \exists z \sigma(x+1,z) )\big)\\
    \theta_4(\sigma) &:= \forall x \exists y\Big(\mathsf{True}_{\Sigma_{m-1}}(\sigma(x,y))  \wedge \mathsf{Pr}_y\big(\forall z \big( \mathsf{True}_{\Sigma_{m-1}}(\sigma(x+1,z))\to m\mathsf{Con}(z)\big) \Big)
\end{align*}

Yet we will show that $\mathsf{B\Sigma_1}\vdash \neg \mathsf{DS}$. By G\"{o}del's second incompleteness theorem, it suffices to prove $\mathsf{Con}(\mathsf{B\Sigma_1}+\mathsf{DS})$ in $\mathsf{B\Sigma_1}+ \mathsf{DS}$. 

\emph{So let's reason in $\mathsf{B\Sigma_1}+ \mathsf{DS}$.}

Let $T_n$ be the $n$-th slice of $\tau$. Observe that $T_0$ proves the $m$-consistency of $T_1$.

Now we consider the sequence $\sigma^\star$ that results from shifting all entries in $\sigma$ one to the left. That is:
$$\forall x \forall y \big( \sigma^\star(x,y) \leftrightarrow \sigma(x+1,y) \big).$$

\begin{claim}
    $T_0+\mathsf{DS}$ is consistent.
\end{claim}

Note that $\theta_2(\sigma^\star)$, $\theta_3(\sigma^\star)$, and $\theta_4(\sigma^\star)$ are true $\Pi_m$ claims (note that here we are using the assumption that $m\geq 3$). 
$$\theta_2(\sigma^\star)\wedge \theta_3(\sigma^\star)\wedge \theta_4(\sigma^\star) $$
is true $\Pi_{m}$. Since $T_0$ is $m$-consistent, $T_0$ is consistent with any true $\Pi_{m}$ claim, so we infer that $T_0$ is consistent with this conjunction.

However, we know from the previous claim that $T_0$ proves the $m$-consistency of $T_1$. Hence $T_0\vdash \theta_1(\sigma^*)$. So $T_0$ is consistent with the conjunction:
$$\theta_1(\sigma^\star)\wedge\theta_2(\sigma^\star)\wedge \theta_3(\sigma^\star)\wedge \theta_4(\sigma^\star) $$
whence $T_0+\mathsf{DS}$ is consistent.

\begin{claim}
$\mathsf{B\Sigma_1}+\mathsf{DS}$ is consistent.
\end{claim}
This follows from the previous claim since $\theta_2(\sigma)$ informs us that $T_0$ proves $\mathsf{B\Sigma_1}$.
\end{proof}


\section{The Index Encoding}\label{non-uniform-section}

In this section we turn to the index encoding. Note that all theorems stated in this section are stated relative to the index encoding of sequences. However, we must now be careful about whether we are focusing on the \emph{uniform} or \emph{non-uniform} versions of our main question. Let's turn to the uniform version first; then we turn to the non-uniform version.

\subsection{The Uniform Version}
For the uniform version of our result we will present both a positive and a negative result.

\subsubsection{A Positive Result}

The theorems in the previous section all encoded sequences of theories via slices rather than via indices. Can we extract a uniform result using the index encoding \emph{from} analogous results using the slice encoding? It is not obvious how to extract one from the theorem statements, but we can extract one from the proofs. The key is that in each case our fixed point definition was piecewise: If this condition obtains, add this r.e.\ set of formulas; if that condition obtains, add this other r.e.\ set of formulas. In both cases the r.e.\ sets of formulas are specified via indices. This yields a piecewise definition of a sequence of indices. However, this piecewise definition involves conditioning, which raises the complexity of the sequence's definition. These observations yield the following theorem:

\begin{theorem}\label{uni-index-pos}
    For each $m\geq 0$ and for each r.e.\ $T\supseteq\mathsf{B\Sigma}_m$ such that $\mathsf{I\Sigma}_m + \forall x m\mathsf{Con}^x(T)$ is $m$-consistent, there is a $\Pi_m$-definable uniform $m$-sequence over $T$.
\end{theorem}

\begin{proof}
For $m=0$, the statement holds by Theorem \ref{fss-theorem} (Friedman-Smoryński-Solovay). Though Theorem \ref{fss-theorem} is stated over $\mathsf{PA}$, the same proof generalizes to this case.

Now let's consider the case $m\geq 1$. We shall adapt the proof of Theorem \ref{craig-trick-theorem} to the indices case. Let $t(z,x)$ be (as in the proof of Theorem \ref{craig-trick-theorem}) a primitive recursive function which, when applied to a code $z$ of a formula $\phi(v,w)$ and a number $x$ returns the code of a formula $\phi(x+1,w)$. Recall that $T$ is an elementary formula defining a set of axioms. Let $e$ be the code of the Turing machine which lists the elements of $T$. Let $\mathsf{m}^{\omega}T$ be $\mathsf{I\Sigma}_m + \forall x \, m\mathsf{Con}^x(T)$ (as in the proof of Theorem \ref{thm::SigmaDefUniSeq}). Finally, let $f(x,w,z)$ be a primitive recursive function which returns the $w$-th code of a Turing machine $M$ such that, on input $x,z$, $M$ outputs $\ulcorner\mathsf{B\Sigma_m} + m\mathsf{Con}(\mathsf{B\Sigma_m} + t(z,x))\urcorner$.

    Let 
    {\footnotesize 
    \begin{align*}
        &\theta(x,s):=\forall k\leq x \; \forall A\leq x \Big( A\in\Pi_{m-1} \wedge \mathsf{Prf}_{m^{\omega}T}(k,\exists x A) \to  \mathsf{True}_{\Pi_{m-1}}\big(A(s_k)\big)\Big).\\
        &\psi(x,y,w,z,s):= \theta(x,s)\wedge y= f(x,w,z)\\
        &\mu y. \exists s,w<y \; \psi(x,y,w,z,s):= \exists s,w <y\big(\psi(x,y,w,z,s) \wedge \forall y'<y \;  \forall s,w <y' \neg \psi(x,y',w,z,s)\big)
    \end{align*}
    }
   Observe that $\theta(x,s)$ is a $\Pi_{m-1}$ formula and $\mu y. \exists s,w<y \; \psi(x,y,w,z,s)$ is $\mathsf{B\Sigma_m}$-provably equivalent to a $\Pi_m$ formula. Put
    \[\tau'(x,y,z):= \mu y. \exists s,w<y \; \psi(x,y,w,z,s)\vee (\forall s\neg\theta(x,s)\wedge y=e).\]
    Since we shall be working over $\mathsf{B\Sigma_m}$, we can assume that $\tau'(x,y,z)$ is a $\Pi_m$ formula. By the diagonal lemma we can fix a formula $\tau(x,y)$ such that provably in Robinson's arithmetic $\mathsf{Q}$
    \[\tau(x,y) \equiv \tau'(x,y,\ulcorner\tau(x,y)\urcorner).\]
    By the canonical proof of the diagonal lemma, we can assume that the complexity of $\tau(x,y)$ is $\Pi_m$. 

    Note that the definition of $\tau'$ is disjunctive; so in one case, it defines one theory, in another case, another theory. By considering cases, for each $n$ there is an index of a total Turing Machine $e_n$ such that $\tau(n,e_n)$ holds. Let $T_n$ be the theory computed by the index $e_n$ (recall that a theory is an axiom set and not its deductive closure). Observe first that $T_n$ contains $\mathsf{B\Sigma_m}$. We argue that for each $n$, $T_n\vdash \exists ! x \tau(n+1,x)$ (note that the uniqueness claim follows trivially from the definition of $\tau'$, but what matters is existence). Fix $n$ and reason in $T_n$. If $\forall s\neg\theta(n+1,s)$, then, by the definition of $\tau'$, we have $\forall x(\tau(n+1,x) \leftrightarrow x=e)$. Otherwise, let $A_1(x_1), \ldots, A_k(x_k)$ be all the $\Pi_{m-1}$-sentences provable in $m^{\omega}T$ whose proofs are below $n+1$. By the assumption that there is an $s$ such that $\bigwedge_{i\leq k} A_i(s_i)$ holds and by $\Pi_{m-1}$ induction we can choose the least such $s$. Now let $a$ be the least code of a Turing Machine which is greater than $s$ and for some $w<a$ satisfies $f(n+1,w,\ulcorner\tau(x,y)\urcorner)$. So in either case, the existence claim is verified.
    
    The proof that $T_0$ is $m$-consistent is exactly as in the slices case (proof of Theorem \ref{thm::SigmaDefUniSeq}). Likewise, since by assumption $\mathsf{m}^{\omega}T$ is $\Sigma_m$-sound, we conclude that for each $n$, $T_n$ proves
    \[\forall x \bigl(\tau(n+1,x)\rightarrow m\mathsf{Con}(x)\bigr).\]
    We note that the above sentence is $\Pi_{m+1}$. Hence by induction we can prove that for each $n$, $T_n$ is $m$-consistent.
\end{proof}

\subsubsection{A Negative Result}

\begin{theorem}\label{new-uni-negative}
    For each $m\geq 2$, there is no $\Sigma_m$-definable uniform $m$-sequence over $\mathsf{B\Sigma_1}$.
\end{theorem}

Note that this result immediately generalizes to extensions of $\mathsf{B\Sigma_1}$ since any $m$-sequence over an extension of $\mathsf{B\Sigma_1}$ is also an $m$-sequence over $\mathsf{B\Sigma_1}$.

\begin{proof}
Suppose that there is such a sequence. Then $\mathsf{DS}$ is true, where $\mathsf{DS}$ is the statement 
$$\exists \sigma \in \Sigma_m \big( \theta_1(\sigma) \wedge \theta_2(\sigma) \wedge \theta_3(\sigma)\wedge \theta_4(\sigma)\big)$$ where we have:
\begin{flalign}
    \theta_1(\sigma) &:= \exists x\big(\mathsf{True}_{\Sigma_m}(\sigma(0,x))\wedge m\mathsf{Con}(x)\big)\\
    \theta_2(\sigma) &:=  \forall x \forall y \big(\mathsf{True}_{\Sigma_m}(\sigma(x,y)) \to \mathsf{Pr}_y(\mathsf{B\Sigma_1})\big)\\
    \theta_3(\sigma) &:=  \forall x \forall y \big(\mathsf{True}_{\Sigma_m}(\sigma(x,y)) \to \mathsf{Pr}_y(  \exists z \sigma(x+1,z) )\big)\\
    \theta_4(\sigma) &:= \forall x \forall y\Big(\mathsf{True}_{\Sigma_m}(\sigma(x,y))  \to \mathsf{Pr}_y\big(\forall z \big( \mathsf{True}_{\Sigma_m}(\sigma(x+1,z))\to m\mathsf{Con}(z)\big) \Big)
\end{flalign}

Yet we will show that $\mathsf{B\Sigma_1}\vdash \neg \mathsf{DS}$. By G\"{o}del's second incompleteness theorem, it suffices to prove $\mathsf{Con}(\mathsf{B\Sigma_1}+\mathsf{DS})$ in $\mathsf{B\Sigma_1}+ \mathsf{DS}$. 

\emph{So let's reason in $\mathsf{B\Sigma_1}+ \mathsf{DS}$.}

Let's introduce the name $T_0$ for some witness to $\theta_1(\sigma)$. $T_0$ is $m$-consistent. From $\theta_3(\sigma)$, we know that $T_0$ proves the $\Sigma_m$ claim $\exists z \sigma(1,z)$. Since $T_0$ is $m$-consistent, this claim must be true. So let's introduce the name $T_1$ for some witness to the claim $\exists y \sigma(1,y)$.

\begin{claim}
    $T_0+\sigma(1,T_1)$ proves the $m$-consistency of $T_1$.
\end{claim}

Consider the instance of $\theta_4(\sigma)$ where $x=0$: 
$$ \forall y\Big(\mathsf{True}_{\Sigma_m}(\sigma(0,y)) \to  \mathsf{Pr}_y\big(\forall z \big( \mathsf{True}_{\Sigma_m}(\sigma(1,z)) \to m\mathsf{Con}(z)) \big).$$
Since we have $\sigma(0,T_0)$ and $\sigma(1,T_1)$, we infer:
$$ \mathsf{Pr}_{T_0}\big(m\mathsf{Con}(T_1)\big).$$
That is, $T_0+\sigma(1,T_1)$ proves the $m$-consistency of $T_1$.

Now we consider the sequence $\sigma^\star$ that results from shifting all entries in $\sigma$ one to the left. That is:
$$\forall x \forall y \big( \sigma^\star(x,y) \leftrightarrow \sigma(x+1,y) \big).$$

\begin{claim}
    $T_0+\mathsf{DS}$ is consistent.
\end{claim}

Note that $\theta_2(\sigma^\star)$, $\theta_3(\sigma^\star)$, and $\theta_4(\sigma^\star)$ are true $\Pi_m$ claims (note that here we are using the assumption that $m\geq 2$). On the other hand, $\sigma^\star(0,T_1)$ is true $\Sigma_m$. So the conjunction:
$$\sigma^\star(0,T_1)\wedge \theta_2(\sigma^\star)\wedge \theta_3(\sigma^\star)\wedge \theta_4(\sigma^\star) $$
is true $\Sigma_{m+1}$. Since $T_0$ is $m$-consistent, $T_0$ is consistent with any true $\Sigma_{m+1}$ claim, so we infer that $T_0$ is consistent with this conjunction.

However, we know from the previous claim that $T_0+\sigma^*(0,T_1)$ proves the $m$-consistency of $T_1$. Yet $\sigma^\star(0,T_1) \wedge m\mathsf{Con}(T_1)$ jointly entail that $T_1$ witnesses $\theta_1(\sigma^\star)$. So $T_0$ is consistent with the conjunction:
$$\theta_1(\sigma^\star)\wedge\theta_2(\sigma^\star)\wedge \theta_3(\sigma^\star)\wedge \theta_4(\sigma^\star) $$
whence $T_0+\mathsf{DS}$ is consistent.

\begin{claim}
$\mathsf{B\Sigma_1}+\mathsf{DS}$ is consistent.
\end{claim}
This follows from the previous claim since $\theta_m(\sigma)$ informs us that $T_0$ proves $\mathsf{B\Sigma_1}$.
\end{proof}

\subsection{The Non-uniform Version}

We now prove our final main theorem. 

\begin{theorem}\label{main-negative}
    For $m\geq 2$, no $\Sigma_1$ non-uniform $m$-sequence over $\mathsf{B\Sigma_1}$ is pointwise provably inhabited.
\end{theorem}

\begin{proof}
Suppose that there is such a sequence. Then $\mathsf{DS}$ is true, where $\mathsf{DS}$ is the statement 
$$\exists \sigma \in \Sigma_m \big( \theta(\sigma) \wedge \theta_2(\sigma) \wedge \theta_3(\sigma)\wedge \theta_4(\sigma)\big)$$ where we have:
\setcounter{equation}{0}
\begin{flalign}
    \theta_1(\sigma) &:= \exists x\big(\mathsf{True}_{\Sigma_m}(\sigma(0,x))\wedge m\mathsf{Con}(x)\big)\\
    \theta_2(\sigma) &:=  \forall x \forall y \big(\mathsf{True}_{\Sigma_m}(\sigma(x,y)) \to \mathsf{Pr}_y(\mathsf{B\Sigma_1})\big)\\
    \theta_3(\sigma) &:=  \forall x \forall y \big(\mathsf{True}_{\Sigma_m}(\sigma(x,y)) \to \mathsf{Pr}_y(  \exists z \sigma(x+1,z) )\big)\\
    \theta_4(\sigma) &:= \forall x \forall y\Big(\mathsf{True}_{\Sigma_m}(\sigma(x,y)) \to \forall z \big( \mathsf{True}_{\Sigma_m}(\sigma(x+1,z)) \to \mathsf{Pr}_y\big(m\mathsf{Con}(z)) \big) 
\end{flalign}

Yet we will show that $\mathsf{B\Sigma_1}\vdash \neg \mathsf{DS}$. By G\"{o}del's second incompleteness theorem, it suffices to prove $\mathsf{Con}(\mathsf{B\Sigma_1}+\mathsf{DS})$ in $\mathsf{B\Sigma_1}+ \mathsf{DS}$. 

\emph{So let's reason in $\mathsf{B\Sigma_1}+ \mathsf{DS}$.}

Let's introduce the name $T_0$ for some witness to $\theta_1(\sigma)$. $T_0$ is $m$-consistent. From $\theta_3(\sigma)$, we know that $T_0$ proves the $\Sigma_m$ claim $\exists z \sigma(1,z)$. Since $T_0$ is $m$-consistent, this claim must be true. So let's introduce the name $T_1$ for some witness to the claim $\exists y \sigma(1,y)$.

\begin{claim}
    $T_0$ proves the $m$-consistency of $T_1$.
\end{claim}

Consider the instance of $\theta_4(\sigma)$ where $x=0$: 
$$ \forall y\Big(\mathsf{True}_{\Sigma_m}(\sigma(0,y)) \to \forall z \big( \mathsf{True}_{\Sigma_m}(\sigma(1,z)) \to \mathsf{Pr}_y\big(m\mathsf{Con}(z)) \big).$$
Since we have $\sigma(0,T_0)$ and $\sigma(1,T_1)$, we infer:
$$ \mathsf{Pr}_{T_0}\big(m\mathsf{Con}(T_1)\big).$$
That is, $T_0$ proves the $m$-consistency of $T_1$.

Now we consider the sequence $\sigma^\star$ that results from shifting all entries in $\sigma$ one to the left. That is:
$$\forall x \forall y \big( \sigma^\star(x,y) \leftrightarrow \sigma(x+1,y) \big).$$

\begin{claim}
    $T_0+\mathsf{DS}$ is consistent.
\end{claim}

Note that $\theta_2(\sigma^\star)$, $\theta_3(\sigma^\star)$, and $\theta_4(\sigma^\star)$ are true $\Pi_m$ claims. On the other hand, $\sigma^\star(0,T_1)$ is true $\Sigma_m$. So the conjunction:
$$\sigma^\star(0,T_1)\wedge \theta_2(\sigma^\star)\wedge \theta_3(\sigma^\star)\wedge \theta_4(\sigma^\star) $$
is true $\Sigma_{m+1}$. Since $T_0$ is $m$-consistent, $T_0$ is consistent with any true $\Sigma_{m+1}$ claim, so we infer that $T_0$ is consistent with this conjunction.

However, we know from the previous claim that $T_0$ proves the $m$-consistency of $T_1$. Yet $\sigma^\star(0,T_1) \wedge m\mathsf{Con}(T_1)$ jointly entail that $T_1$ witnesses $\theta_1(\sigma^\star)$. So $T_0$ is consistent with the conjunction:
$$\theta_1(\sigma^\star)\wedge\theta_2(\sigma^\star)\wedge \theta_3(\sigma^\star)\wedge \theta_4(\sigma^\star) $$
whence $T_0+\mathsf{DS}$ is consistent.

\begin{claim}
$\mathsf{B\Sigma_1}+\mathsf{DS}$ is consistent.
\end{claim}
This follows from the previous claim since $\theta_2(\sigma)$ informs us that $T_0$ proves $\mathsf{B\Sigma_1}$.
\end{proof}

\bibliographystyle{plain}
\bibliography{bibliography}

\end{document}